\documentclass[12pt,leqno]{amsart}
\usepackage{amsmath,amssymb,amsfonts, amscd,   
pifont, amsthm,  amsxtra, latexsym}
\usepackage[all]{xy}
\usepackage{here}
\usepackage[dvips]{graphicx}

\usepackage{psfrag}
\usepackage{amsmath}
\usepackage{amsthm}
\usepackage{amsfonts}
\usepackage{amssymb}
\theoremstyle{plain}
\newtheorem{thm}{Theorem}[section]
\newtheorem{theorem}[thm]{Theorem}

\newtheorem{corollary}[thm]{Corollary}
\newtheorem{proposition}[thm]{Proposition}

\newtheorem{lemma}[thm]{Lemma}

\newtheorem{fact}[thm]{Fact}

\newtheorem{remark}[thm]{Remark}

\theoremstyle{definition}

\newtheorem{example}[thm]{Example}

\makeatletter
\@addtoreset{equation}{section}
\def\@seccntformat#1{\csname the#1\endcsname.\quad}
\def\@seccntformat#1{\csname the#1\endcsname%
\expandafter\ifx\csname#1\endcsname\subsection.\fi\quad}
\makeatother
\newlength{\barlength}
\newcommand{\minusfill}{$\mathsurround=0pt\mathord- \mkern-6mu
   \cleaders\hbox{$\mkern-3mu \mathord- \mkern-3mu$}\hfill
     \mkern-6mu \mathord-$}
\newcommand{\yokobo}{\hbox to 1.2em{\minusfill}}
\newcommand{\equalfill}{$\mathsurround=0pt\mathord= \mkern-6mu
    \cleaders\hbox{$\mkern-3mu \mathord= \mkern-3mu$}\hfill
       \mkern-6mu \mathord=$}
\newcommand{\Longlongrightarrow}
       {\hbox to 2em{\equalfill$\mkern-3mu\Rightarrow$}}
\newcommand{\Longlongleftarrow}
       {\hbox to 2em{$\Leftarrow\mkern-3mu$\equalfill}}
\newcommand{\tsume}{\kern-.35em}
\newlength{\circlength}
\divide\circlength by 2
\advance\circlength by 0.1em

\newcommand\Sol{{\mathcal {S}}\!{\text{\it{ol}}}\,}
\begin{document}
\title{A generating operator for Rankin--Cohen brackets}
\author{Toshiyuki Kobayashi and
Michael Pevzner}

%\date{\today}

\maketitle

\begin{abstract}
Motivated by the classical ideas of generating functions
 for orthogonal polynomials, 
 we initiate a new line of
 investigation
 on \lq\lq{generating operators}\rq\rq\
 for a family of differential operators
between two manifolds.  
We prove a novel formula of the generating operators
 for the Rankin--Cohen brackets
 by using higher-dimensional residue calculus.  
Various results on the generating operators are also explored from the perspective of infinite-dimensional representation theory.  
\end{abstract}


\medskip
\noindent
\textit{2020 MSC}:
Primary
22E45, % (1973-now) Representations of Lie and linear algebraic groups over real fields: analytic methods {For the purely algebraic theory, see 20G05}
47B38%  (1980-now) Linear operators on function spaces (general),
;
Secondary
11F11, %  (1980-now) Holomorphic modular forms of integral weight
32A27, % (1980-now) Residues for several complex variables [See also 32C30]
30H10, %  (2010-now) Hardy spaces [See also 42B30, 46E30]
33C45, %  (1991-now) Orthogonal polynomials and functions of
%hypergeometric type (Jacobi, Laguerre, Hermite, Askey scheme, etc.)
%{For general orthogonal polynomials and functions, see also 42C05}
43A85.%  (1973-now) Harmonic analysis on homogeneous spaces
%22E45, 22E46, 32M15, 33C45, 33C80, 43A85.
%22E46 (1980-now) Semisimple Lie groups and their representations
%32M15  (1973-now) Hermitian symmetric spaces, bounded symmetric domains, Jordan algebras (complex-analytic aspects) [See also 22E10, 22E40, 53C35, 57T15]
%33C80  (1991-now) Connections of hypergeometric functions with groups and algebras, and related topics

% \noindent{ Declarations of interest: none.}
\section{Introduction}
\label{sec:Intro}

To any sequence  $\{a_\ell\}_{\ell\in\mathbb N}$ one may associate a
formal power series such as $\underset{\ell=0}{\overset{\infty}{\sum}}
a_\ell {t^\ell}$ or $\underset{\ell=0}{\overset{\infty}{\sum}}
a_\ell\frac{t^\ell}{\ell!}$ .  The resulting \emph{generating functions} are fascinating objects providing powerful tools
for studying various combinatorial problems when $a_\ell$ are integers or, more generally,
 polynomials. 
One may
quantize this construction by considering differential operators as
non-commutative analogues of polynomials
 and may study the resulting 
 \lq\lq{generating operators}\rq\rq. 
Dealing with the sequence of differential
operators given by iterated powers of some remarkable operator yields
the notion of an operator semigroup which is nowadays a classical tool
for the spectral theory of unbounded operators 
({\it{e.g.}} the Hille--Yosida theory). 
%Inspired by the ideas of deformation quantization 
 %\cite{xKon03, xP18}, 
We explore yet another direction by
introducing a sequence of differential operators
 with a different
algebraic structure
 which is not defined by one single operator anymore.

Let us start with our general setting.  
Suppose that $\Gamma(X)$ and $\Gamma(Y)$ are the spaces
 of functions 
 on $X$ and $Y$, 
 respectively.  
Given a family of linear operators
 $R_{\ell}\colon \Gamma(X) \to \Gamma(Y)$, 
 we consider a formal power series
\begin{equation}
\label{eqn:Rlt}
T\equiv T(\{R_{\ell}\}; t):=\sum_{\ell=0}^{\infty}\frac{R_{\ell}}{\ell!}t^{\ell} \in \operatorname{Hom}(\Gamma(X), \Gamma(Y)) \otimes {\mathbb{C}}[[t]]. 
\end{equation}

When $X=\{\text{point}\}$, 
 $R_{\ell}$ is identified with an element of $\Gamma(Y)$, 
 and such a formal power series is called a {\it{generating function}}, 
 which has been particularly prominent
 in the classical study of orthogonal polynomials
 for $\Gamma(Y)={\mathbb{C}}[y]$, 
 see {\it{e.g.,}} \cite{xSrMa84, xSz75}.

When $X=Y$, 
 $\operatorname{Hom}(\Gamma(X), \Gamma(Y))\simeq {\operatorname{End}}(\Gamma(X))$
 has a ring structure
 and one may take $R_{\ell}$
 to be the $\ell$-th power of a {\it{single}} operator $R$
 on $X$.  
In this case, 
 the operator $T$ in \eqref{eqn:Rlt}
 may be written as $e^{t R}$
 if the summation converges.  
We note that even if $R$ is a differential operator 
 on a manifold $X$, 
 the resulting operator $T=e^{t R}$ is not a differential operator
 any more in general .  
For example, 
 if $R=\frac {d}{d z}$ acting on ${\mathcal{O}}({\mathbb{C}})$, 
 then $T=e^{t \frac d {dz}}$ is the shift operator
 $f(z) \mapsto f(z+t)$.  
For a self-adjoint operator $R$
 with bounded eigenvalues from the above, 
 the operator $T$ has been intensively studied
 as the {\it{semigroup}} $e^{t R}$
 generated by $R$ for $\operatorname{Re} t>0$:
 typical examples include 

\par\indent
$\bullet$\enspace
 the heat kernel
 for $R=\Delta$, 

\par\indent
$\bullet$\enspace
the Hermite semigroup
 for $R=\frac 1 4 (\Delta-|x|^2)$  on $L^2({\mathbb{R}}^n)$, 

\par\indent
$\bullet$\enspace
the Laguerre semigroup
 for $R=|x|(\frac{\Delta}4-1)$ on $L^2({\mathbb{R}}^n, \frac{1}{|x|} d x)$.

Let us consider a more general setting
 where we allow $X \ne \{\text{point}\}$ and $X \ne Y$.  
In this generality, 
 we refer to $T$ in \eqref{eqn:Rlt}
 as the {\it{generating operator}}
 for a family of operators
 $R_{\ell} \colon \Gamma(X) \to \Gamma(Y)$.

In the present work
 we initiate a new line of investigation of 
\lq\lq{generating operators}\rq\rq\
 in the setting
 that $(X,Y)=({\mathbb{C}}^2, {\mathbb{C}})$
 and that $\{R_{\ell}\}$ are the Rankin--Cohen brackets
\cite{xCo75, xRa56}.  
We shall find a closed formula 
 of the generating operator $T$ as an integral operator, 
 through which we explore its basic properties
 and various aspects.  
\vskip 3pc

It is known
 that covariant differential operators are often obtained as residues
of a meromorphic family of integral transformations. 
For instance, the iterated powers of the Dirac operator are
 the residues of the meromorphic family
 of the Knapp--Stein intertwining operators, 
 see {\it{e.g.,}} a recent paper \cite{CLOR20}.

The inverse direction is more involved.  In fact, 
 some covariant differential operators cannot be obtained as residues, 
 which are referred to as {\it{sporadic operators}}.  
One of the important applications of the \emph{generating operator} introduced in this article provides us a method
 to go in the inverse direction, 
 namely,
 to construct a meromorphic family of non-local symmetry breaking operators
 out of discrete data.  
In the subsequent paper \cite{K23}, 
 we give a toy model
 which constructs various fundamental operators
 such as invariant trilinear forms
 on infinite-dimensional representations, 
 the Fourier and the Poisson transforms
 on the anti-de Sitter space, 
 and non-local symmetry breaking operators
 for the fusion rules 
 among others, 
 out of just countable data of the Rankin--Cohen brackets, 
 for which the key of the proof is
 the explicit formula \eqref{eqn:Luminy230314}
 of the {\it{generating operator}}
proved in this article.

\vskip 1pc

The article is organized as follows.  
In Section \ref{sec:T} we give an integral expression of 
 the \lq\lq{generating operator}\rq\rq\ $T$
 of the Rankin--Cohen bidifferential operators
 (Theorem \ref{thm:gRC}), 
 and discuss the domain of holomorphy.  
In Section \ref{sec:P}, 
 we introduce a second-order differential operator $P$
 on ${\mathbb{C}}^2$
 which plays a key role 
 in the detailed analysis of $T$
 (Theorems \ref{thm:Luminy0316} and \ref{thm:Tlsol}).  
In Section \ref{sec:Hardy}
 we focus on operators between Hilbert spaces, 
 and prove
 that $T$ gives rise to a natural decomposition 
 of the completed tensor product of two Hardy spaces
 (Theorem \ref{thm:Hardy}).  
In Section \ref{sec:SL}
 we discuss briefly different perspectives
 of the generating operator $T$ from the viewpoint of unitary representation theory
 of real reductive groups, 
 in particular, from that of symmetry breaking operators
 and holographic operators
 associated with branching problems 
 ({\it{fusion rules}})
 for $SL(2,{\mathbb{R}})$.  
\vskip 1pc

Notation.\enspace
${\mathbb{N}}=\{0,1,2,\cdots\}$, 
${\mathbb{R}}_+=\{x \in {\mathbb{R}}: x>0\}$.  

%%%%%%%%%%%%%%%%%%%%%%%%%%%%%%%%%%%%%%%
\section{Basic properties of the integral operator $T$}
\label{sec:T}
%%%%%%%%%%%%%%%%%%%%%%%%%%%%%%%%%%%%%%%

Let $D$ be an open set in ${\mathbb{C}}$.  
For a holomorphic function $f(\zeta_1, \zeta_2)$ in $D \times D$, 
 we introduce an integral transform by 
\begin{equation}
\label{eqn:Luminy230314}
(T f)(z,t):=
\frac{1}{(2 \pi \sqrt{-1})^2}\oint_{C_1}\oint_{C_2} 
\frac{f(\zeta_1, \zeta_2)}{(\zeta_1-z)(\zeta_2-z)+t(\zeta_1 - \zeta_2)}d \zeta_1 d \zeta_2, 
\end{equation}
where $C_j$ are contours in $D$
 around the point $z$ ($j=1,2$).  
The denominator will be denoted by 
\begin{equation}
\label{eqn:Q}
Q \equiv Q(\zeta_1, \zeta_2;z, t):=(\zeta_1-z)(\zeta_2-z)+t(\zeta_1-\zeta_2).  
\end{equation}

We note that the denominator is an irreducible polynomial of $\zeta_1$ and $\zeta_2$ when $t \ne 0$. 
We shall give closed formulas of the transform $T f(z,t)$
 for a family of meromorphic functions $f(\zeta_1, \zeta_2)$, corresponding to the minimal
 $K$-types in representation theory, see Example \ref{ex:23050410}.

We begin with general properties of the operator $T$.  
\begin{theorem}
\label{thm:23032333}
~~~
\newline
{\rm{(1)}}\enspace
There exists an open neighbourhood $U$ of $D \times \{0\}$
 in ${\mathbb{C}}^2$ 
such that $T \colon {\mathcal{O}}(D \times D) \to {\mathcal{O}}(U)$
 is well-defined.  
\par\noindent
{\rm{(2)}}\enspace
$T f(z,0)=f(z,z)$ for any $z \in D$.  
\par\noindent
{\rm{(3)}}\enspace
For any neighbourhood $U$ of $D \times \{0\}$ in ${\mathbb{C}}^2$, 
 $T$ is injective.  
\end{theorem}

\begin{proof}[Proof of (1) and (2) in Theorem \ref{thm:23032333}]
(1)\enspace
For $z \in D$ and $t \in {\mathbb{C}}$, 
 we define an analytic set by 
\[
 {\mathcal{N}}_{z,t}:=
 \{(\zeta_1, \zeta_2)\in D \times D:
  Q(\zeta_1, \zeta_2;z,t)=0\}.  
\]
Then there exists a neighbourhood $W$ of $t=0$ such that 
$C_1\times C_2\subset  {\mathcal{N}}_{z,t}$ for all $t\in W$.
The integral \eqref{eqn:Luminy230314} does not change
 if we replace $C_1 \times C_2$ 
 by a compact surface $S$, 
 as far as $S$ belongs to the same second homology class
in $D \times D \setminus {\mathcal{N}}_{z,t}$.  
We define $d(z)\equiv d(z,\partial D)$
 to be the distance from $z$
 to the boundary $\partial D$.  
We set $d(z):=\infty$ if $\partial D=\emptyset$, 
 namely, 
 if $D={\mathbb{C}}$.  
We claim that $T f (z,t)$ is well-defined
 and holomorphic in 
\begin{equation}
\label{eqn:UD}
  U_D:=\{(z,t) \in D \times {\mathbb{C}}:
  2|t|< d(z, \partial D)\}.  
\end{equation}

We fix $z\in D$, and set $R:=d(z,\partial D)$. Let $\varepsilon>0$. 
 If we take $C_1=C_2$ 
 to be the circle of radius $R(1-\varepsilon)$
 centered at $z$, 
 then $(C_1 \times C_2) \cap  {\mathcal{N}}_{z',t}=\emptyset$
 for any $(z',t)$ satisfying $|z'-z|< R \varepsilon$
 and $2|t|<R(1-3 \varepsilon)$
 because
\[
  |(\zeta_1-z')(\zeta_2-z')| 
  >
  R^2(1-2\varepsilon)^2
  >
  R^2(1-\varepsilon)(1-3 \varepsilon)
  > 
  |t(\zeta_1-\zeta_2)|.  
\]
Therefore, $(C_1\times C_2)\cap {\mathcal{N}}_{z',t}=\emptyset$, hence
 $T f(z',t)$ is holomorphic 
 in this region.  
Taking the limit as $\varepsilon \to 0$, 
 we conclude 
 that $T f$ is well-defined 
 and holomorphic in the open neighbourhood 
 of $\{z\}\times\{ t\in\mathbb C : 2\vert t\vert< d(z,\partial D)\}$ for every $z\in D$, hence it is
 holomorphic in $U_D$.  
\newline
(2)\enspace
Clear from Cauchy's integral formula.  
\end{proof}

\begin{example}
\label{ex:UD}
We make explicit two important examples of the domains $U_D$ introduced in \eqref{eqn:UD}.
\newline
{\rm{(1)}}\enspace
$U_D={\mathbb{C}} \times {\mathbb{C}}$
 if $D={\mathbb{C}}$.  
\newline
{\rm{(2)}}\enspace
$U_D=
  \{(z, t) \in {\mathbb{C}}^2: 2|t|<\operatorname{Im} z\}
$
if $D$ is the upper half plane
$
  \Pi:=\{\zeta \in {\mathbb{C}}: \operatorname{Im} \zeta >0\}.  
$
\end{example}

Before giving a proof of the third statement 
 of Theorem \ref{thm:23032333}, 
 we show
 that $T$ is a \lq\lq{generating operator}\rq\rq\
 for the family of the Rankin--Cohen brackets.  
For $\ell \in {\mathbb{N}}$ we define 
$
  R_{\ell} \colon {\mathcal{O}}(D \times D) \to 
  {\mathcal{O}}(D)
$,
 $f(\zeta_1, \zeta_2) \mapsto (R_{\ell}f)(z)$
 by 
\begin{equation}
\label{eqn:RC}
  R_{\ell} f(z):=\sum_{j=0}^{\ell} (-1)^j \left(\ell \atop j\right)^2 
 \left. \frac{\partial^{\ell} f(\zeta_1, \zeta_2)}{\partial \zeta_1^{\ell-j} \partial \zeta_2^j}\right|_{\zeta_1=\zeta_2=z}.  
\end{equation}

\begin{theorem}
[generating operator of the Rankin--Cohen brackets]
\label{thm:gRC}
The integral operator $T$
 in \eqref{eqn:Luminy230314} 
 is expressed as 
\[
   T f(z,t)=\sum_{\ell =0}^{\infty} \frac{t^{\ell}}{\ell !}R_{\ell} f(z)
\quad
  \text{for any $f \in {\mathcal{O}}(D \times D)$.}  
\]
\end{theorem}

\begin{remark}
For $f(\zeta_1, \zeta_2)=f_1(\zeta_1) f_2(\zeta_2)$
 with some $f_1, f_2 \in {\mathcal{O}}(D)$, 
 $(R_{\ell}f)(z)$ takes the form
$
\sum_{j=0}^{\ell} (-1)^j \left(\ell \atop j\right)^2
\frac{\partial^{\ell-j}f_1(z)}{\partial z^{\ell-j}}
\frac{\partial^{j}f_2(z)}{\partial z^{j}},
$
 which is the Rankin--Cohen bidifferential operator 
 $R_{\lambda', \lambda''}^{\lambda'''}(f_1, f_2)$
 at $(\lambda',\lambda'',\lambda''')=(1,1,2+2\ell)$
 with the notation
 as in \cite[(2.1)]{KP20}.  
\end{remark}

\begin{proof}
[Proof of Theorem \ref{thm:gRC}]
By the first statement of Theorem \ref{thm:23032333}, 
 one can expand $T f(z,t)$
 into the Taylor series of $t$:
\[
   T f(z,t)=\sum_{\ell=0}^\infty t^{\ell} (T_{\ell} f)(z)
\]
with coefficients $T_{\ell}f(z) \in {\mathcal{O}}(D)$.  
Accordingly, 
we expand $Q^{-1}$ into the Taylor series of $t$:
\begin{equation}
\label{eqn:QTaylor}
  \frac 1 Q
  =
  \sum_{\ell=0}^\infty \frac{(-1)^{\ell}(\zeta_1-\zeta_2)^{\ell}t^{\ell}}{(\zeta_1-z)^{\ell+1}(\zeta_2-z)^{\ell+1}}.  
\end{equation}
An iterated use of the Cauchy integral formula
 gives an explicit formula of $(T_{\ell} f)(z)$ by 
\begin{align*}
(T_{\ell} f)(z) =& \frac{(-1)^{\ell}}{(2 \pi \sqrt{-1})^2}
              \oint_{C_1} \oint_{C_2}
              \frac{(\zeta_1-\zeta_2)^{\ell} f (\zeta_1, \zeta_2)}
                   {(\zeta_1-z)^{\ell+1}(\zeta_2-z)^{\ell+1}}
             d \zeta_1 d \zeta_2
\\
=& \frac{(-1)^{\ell}}{2 \pi \sqrt{-1}}
              \oint_{C_2} 
              \frac{(\frac{\partial}{\partial \zeta_1})^{\ell}|_{\zeta_1=z}
((\zeta_1-\zeta_2)^{\ell} f (\zeta_1, \zeta_2))}
                   {\ell ! (\zeta_2-z)^{\ell+1}}
             d \zeta_2
\\
=& \frac{1}{2 \pi \sqrt{-1}}
\sum_{j=0}^{\ell}
\frac{(-1)^j \ell!}
{(j!) ((\ell-j)!)^2}
\oint_{C_2} 
\frac{\frac{\partial^{\ell-j} f}{\partial \zeta_1^{\ell-j}} (z, \zeta_2)}
{(\zeta_2-z)^{j+1}}
             d \zeta_2
\\
=& \sum_{j=0}^{\ell}
              \left. 
              \frac{(-1)^j \ell!}{(j! (\ell-j)!)^2}
              \frac {\partial^{\ell} f(\zeta_1, \zeta_2)}{\partial \zeta_1^{\ell-j}\partial \zeta_2^j}\right|_{\zeta_1=\zeta_2=z}
=\frac 1 {\ell !}(R_{\ell} f)(z).  
\end{align*}

Hence the theorem is shown.  
\end{proof}

We are ready to prove the third statement
 of Theorem \ref{thm:23032333}, 
 which uses the property 
 that the signature of the coefficients in \eqref{eqn:RC} alternates.

\begin{proof}
[Proof of (3) in Theorem \ref{thm:23032333}]
Suppose $T f \equiv 0$ for $f \in {\mathcal{O}}(D \times D)$.  
We set $a_{i,j}(z):=
\left. \frac{\partial^{i+j} f}{\partial \zeta_1^i \partial \zeta_2^j}\right|_{\zeta_1=\zeta_2=z}$.  
We shall prove $a_{i,j}(z)\equiv 0$
 for all $i,j$ by the induction on $k:=i+j$.  
The case $k=0$ is clear 
 because $a_{0,0}(z)=f(z,z)=(T f)(z,0)$.  
Suppose now that $a_{i,j}(z) \equiv 0$ for all $i+j=k$.  
Since $\frac{d}{dz} a_{i,j}(z)= a_{i+1,j}(z)+ a_{i,j+1}(z)$,
 one has $a_{k+1-j,j}+a_{k-j,j+1}=0$
 for all $0 \le j \le k$, 
 namely, 
 $a_{k+1-j,j}=(-1)^j a_{k+1,0}$.  
In turn, 
 $(\frac{\partial}{\partial t})^{k+1}|_{t=0} T f =(\sum_{j=0}^{k+1} \left({k+1} \atop j\right)^2) a_{k+1,0}$.  
Hence $a_{k+1,0}(z) \equiv 0$, 
 and thus $a_{i,j}(z) \equiv 0$ for all $i,j$
 with $i+j=k+1$.  
Therefore,
 the holomorphic function $f(\zeta_1, \zeta_2)$ must be identically zero.  
\end{proof}

%%%%%%%%%%%%%%%%%%%%%%%%%%%%%%%%%%%%%
\section{Differential operator $P$ and the generating operator}
\label{sec:P}
%%%%%%%%%%%%%%%%%%%%%%%%%%%%%%%%%%%%%

The following differential operator
 on ${\mathbb{C}}^2$ plays a key role 
 in the analysis
 of the generating operator $T$.  

\begin{equation}
\label{eqn:23031847}
P:=
(\zeta_1-\zeta_2)^2 
\frac{\partial^2}{\partial \zeta_1 \partial \zeta_2}
-
(\zeta_1-\zeta_2) 
(\frac{\partial}{\partial \zeta_1}-\frac{\partial}{\partial \zeta_2}).  
\end{equation}

The goal of this section is to prove the following:

\begin{theorem}
\label{thm:Luminy0316}
Let $D$ be an open set in ${\mathbb{C}}$.  
For any $f \in {\mathcal{O}}(D \times D)$, 
\[
T(P f)(z,t) = -(t \frac{\partial}{\partial t})(t \frac{\partial}{\partial t}+1)
T f(z,t).  
\]
\end{theorem}

One derives from Theorem \ref{thm:Luminy0316}
 that the set of eigenvalues of $P$ is discrete:

\begin{corollary}
[Eigenvalues of $P$]
\label{cor:Peigen}
Let $D$ be a connected open set in ${\mathbb{C}}$.  
If there is a non-zero function $f \in {\mathcal{O}}(D \times D)$
 satisfying $P f = \lambda f$ for some $\lambda \in {\mathbb{C}}$, 
then $\lambda$ is of the form $-\ell (\ell+1)$
 for some $\ell \in {\mathbb{N}}$.  
\end{corollary}

For $\ell \in {\mathbb{N}}$, 
 we consider the space of all eigenfunctions:
\begin{equation}
\label{eqn:Sol}
  \Sol(D \times D)_{\ell}
  :=
  \{f \in {\mathcal{O}}(D \times D):
   P f = -\ell(\ell+1)f \}.  
\end{equation}

We shall see in Corollary \ref{cor:dimSol}
 that $\Sol(D \times D)_{\ell}$ is infinite-dimensional 
 for any $\ell \in {\mathbb{N}}$
 and for any non-empty open subset $D$.  

\begin{remark}
{\rm{(1)}}\enspace
In Theorem \ref{thm:Hardy}, 
 we shall prove that 
$P$ defines a self-adjoint operator 
 on the completed tensor product of two Hardy spaces.  
\newline
{\rm{(2)}}\enspace
Taking this opportunity, 
 we would like to point out
 that the first term
 of the differential operator
 $P_{\lambda', \lambda''}$
 in \cite[(2.31)]{KP20}
 was wrongly stated:
the correct formula is 
\[
P_{\lambda', \lambda''}
=
(\zeta_1-\zeta_2)^2 \frac{\partial^2}{\partial \zeta_1 \partial \zeta_2}
-
(\zeta_1-\zeta_2)
(\lambda'' \frac{\partial}{\partial \zeta_1}
-\lambda' \frac{\partial}{\partial \zeta_2}). 
\]
All the theorems
 involving $P_{\lambda', \lambda''}$ valid with this definition.  
\end{remark}

\begin{proof}
[Proof of Corollary \ref{cor:Peigen}]
Suppose $P f=\lambda f$ with $f \not \equiv 0$.  
Then $T f \not \equiv 0$
 because $T$ is injective 
 by Theorem \ref{thm:23032333}.  
By Theorem \ref{thm:Luminy0316}, 
 one has 
\[
  \vartheta_t (\vartheta_t+1) T f
  =
  -T(P f)
  =
  -\lambda T f, 
\]
where $\vartheta_t$ denotes the Euler homogeneity operator
 $t \frac{\partial}{\partial t}$.  
We observe 
 that $\vartheta_t (\vartheta_t +1)t^{\ell} = \ell(\ell+1) t^{\ell}$
 for every $\ell \in{\mathbb{N}}$.  
Since $T f(z, t)$ is holomorphic in a neighbourhood of $t=0$, 
 possible eigenvalues
 of $\vartheta_t(\vartheta_t+1)$ are of the form 
 $\ell(\ell+1)$ for some $\ell \in {\mathbb{N}}$, 
 and the corresponding eigenfunctions are of the form
 $t^{\ell}\varphi(z)$ 
for some holomorphic function $\varphi(z) \in {\mathcal{O}}(D)$.  
Thus the corollary is proved.  
\end{proof}

The following statement is clear from the above proof.  

\begin{corollary}
\label{cor:Luminy0316}
Let $\ell \in {\mathbb{N}}$.  
Then the following two conditions on $f \in {\mathcal{O}}(D \times D)$ are equivalent:
\par\noindent
{\rm{(i)}}\enspace
$f \in \Sol(D \times D)_{\ell}$, 
\par\noindent
{\rm{(ii)}}\enspace
$T f (z,t)$ is of the form $t^{\ell} \varphi (z)$
 for some $\varphi \in {\mathcal{O}}(D)$.    
\end{corollary}

The rest of the section is devoted to the proof of Theorem \ref{thm:Luminy0316}
 by comparing the integral expressions
 of $T(P f)$ and $\vartheta_t(\vartheta_t+1)T f$.

The following formula for $\vartheta_t (\vartheta_t+1)T f$ is an immediate consequence
 of the definition \eqref{eqn:Luminy230314} 
 of the generating operator $T$.  
For the rest of the paper, 
 we omit writing the contours $C_1$ and $C_2$ in the integrals for simplicity.  

\begin{lemma}
\label{lem:23031632}
For any $f \in {\mathcal{O}}(D \times D)$, 
 one has
\[
\vartheta_t (\vartheta_t+1) T f(z,t)
=\frac{-2}{(2\pi \sqrt{-1})^2}
\oint \oint
\frac{(\zeta_1-\zeta_2)(\zeta_1-z)(\zeta_2-z)t}{Q^3}f d \zeta_1 d \zeta_2.  
\]
\end{lemma}

It is more involved to find the integral expression of $T(P f)$.  
For this, 
 we set
\begin{align*}
I_j(f)
:=& \frac{1}{(2 \pi \sqrt{-1})^2} \oint\oint \frac {\zeta_j^2 - \zeta_1 \zeta_2}Q 
                     \frac{\partial^2 f}{\partial \zeta_1 \partial \zeta_2} d \zeta_1 d \zeta_2
\quad
\text{for $j=1,2$, }
\\
I_3(f):=& -\frac {1}{(2\pi \sqrt{-1})^2}
       \oint \oint \frac{(\zeta_1-\zeta_2) (\frac{\partial}{\partial \zeta_1} - \frac{\partial}{\partial \zeta_2})f} Q d \zeta_1 d \zeta_2.  
\end{align*}

By the definition \eqref{eqn:23031847} of $P$, 
one has 
\[
   T(P f)= I_1(f)+I_2(f)+I_3(f).  
\]
In view of Lemma \ref{lem:23031632}, 
 Theorem \ref{thm:Luminy0316} will be derived from
 the following two propositions.  

\begin{proposition}
\label{prop:23032266}
Let $\varepsilon (1):=-1$
 and $\varepsilon(2)=1$.  
For $j=1,2$, 
one has 
\begin{equation}
  I_j(f)
 = \frac{1}{(2 \pi \sqrt{-1})^2}
   \oint \oint \frac{\varepsilon(j)(\zeta_j-z)^2 + 2 t (\zeta_1-z)(\zeta_2-z)\zeta_j}{Q^2}  
                f (\zeta_1, \zeta_2)d \zeta_1 d \zeta_2.  
\label{eqn:23032266}
\end{equation}
\end{proposition}

\begin{proposition}
\label{prop:23031775}
\begin{equation*}
I_3(f)= \frac {1}{(2\pi \sqrt{-1})^2}
       \oint \oint \frac{(\zeta_1-z)^2 +(\zeta_2-z)^2}{Q^2} f(\zeta_1, \zeta_2) d \zeta_1 d \zeta_2.  
\end{equation*}
\end{proposition}

For the proof of Propositions \ref{prop:23032266} and \ref{prop:23031775}, 
 we need some preparations.  
We define $\xi_1 \equiv \xi_1(\zeta_2)$ and $\xi_2 \equiv \xi_2(\zeta_1)$
 by 
\begin{equation}
\label{eqn:Xi}
 \xi_1:=\frac{(\zeta_2-z)z+t \zeta_2}
             {\zeta_2-z+t}, 
\qquad
 \xi_2 :=\frac{(\zeta_1-z)z-t \zeta_1}
             {\zeta_1-z-t}.  
\end{equation}
Then one has 
$
   Q(\xi_1, \zeta_2)=Q(\zeta_1, \xi_2)=0, 
$
 where we recall 
\[
  Q \equiv Q(\zeta_1, \zeta_2)
  =(\zeta_1-z)(\zeta_2-z)+t(\zeta_1-\zeta_2).  
\]
We set
$\widetilde{\zeta_1}:=\zeta_1-z-t=\frac{\partial Q}{\partial \zeta_2}$
 and $\widetilde{\zeta_2}:=\zeta_2-z+t=\frac{\partial Q}{\partial \zeta_1}$.  
One has
\[
  Q =\widetilde{\zeta_1} \widetilde{\zeta_2}+t^2
    =\widetilde{\zeta_2}(\zeta_1-\xi_1)=\widetilde{\zeta_1}(\zeta_2-\xi_2).  
\]

We list some convenient formul{\ae}
 which are direct from the definition.  
\begin{alignat}{3}
\label{eqn:230317100}
\xi_1-\zeta_2=&\frac{-(\zeta_2-z)^2}{\widetilde{\zeta_2}}, 
\quad
&&\zeta_1-\xi_1=\frac{Q}{\widetilde{\zeta_2}}, 
\quad
&&\xi_1-z=\frac{t(\zeta_2-z)}{\widetilde{\zeta_2}}.  
\\
\label{eqn:23031810}
\xi_2-\zeta_1=&-\frac{(\zeta_1-z)^2}{\widetilde{\zeta_1}}, 
\quad
&&\zeta_2-\xi_2=\frac{Q}{\widetilde {\zeta_1}}, 
\quad
&&\xi_2-z=\frac{-t(\zeta_1-z)}{\widetilde{\zeta_1}}.  
\end{alignat}

In the one-variable case, 
 the Cauchy integral formula implies
\[
   \frac{1}{\ell !}
   \oint \frac{\varphi^{(k)}(\zeta)}{(\zeta-z)^{\ell+1}} d \zeta
   =
   \frac{1}{(\ell+k) !}
   \oint \frac{\varphi(\zeta)}{(\zeta-z)^{\ell+k+1}} d \zeta
\]
for a holomorphic function $\varphi(\zeta)$
 and for any $\ell, k \in {\mathbb{N}}$.  
However, 
 in our setting,
 since $Q$ is an irreducible polynomial
 of the two variables $\zeta_1$ and $\zeta_2$ for $t \ne 0$, 
 the integration formul{\ae} for derivatives
 of a holomorphic function $F(\zeta_1,\zeta_2)$
 against the integral kernel $Q^{-1}$ 
 or its power 
 are not so simple
 as in the one-variable case.  
We establish such formul{\ae} for derivatives 
 against the integral kernel $\zeta_j^a Q^{-b}$
 ($a,b \in {\mathbb{N}}$) as below.

For $a, b \in {\mathbb{N}}$, 
 we define functions 
 $H_{a,b}(\zeta_1, \zeta_2)$
 inductively by the following recurrence relation
\begin{equation}
\label{eqn:Habdef}
   H_{a,b}:=(\xi_1-z)^a H_{0,b} + \widetilde{\zeta_2}^{-1} Q
   \sum_{i=0}^{a-1} (\xi_1-z)^iH_{a-1-i,b-1}, 
\end{equation}
with initial terms
\begin{equation}
\label{eqn:23031774}
\text{$H_{a,0}:=0$\quad and \quad$H_{0,b}:=b \widetilde{\zeta_2}$}.  
\end{equation}

\begin{lemma}
\label{lem:230326}
For any $a, b \in {\mathbb{N}}$, 
 one has
\begin{equation}
\label{eqn:Hab}
  \oint\frac{(\zeta_1-z)^a}{Q^b}\frac{\partial F}{\partial \zeta_1}d \zeta_1
=
  \oint \frac{H_{a,b}}{Q^{b+1}}F d \zeta_1.  
\end{equation}
\end{lemma}

Analogous formul\ae\ to \eqref{eqn:Hab} hold
 if we replace $\zeta_1$ by $\zeta_2$ and $t$ by $-t$.

\begin{proof}
We begin with the case $a=0$.  
Since $\xi_1 \equiv \xi_1(\zeta_2)$ is independent 
 of the variable $\zeta_1$, 
 one has 
\[
 \oint \frac{1}{(\zeta_1-\xi_1)^b} 
\frac{\partial F} {\partial \zeta_1} d \zeta_1
= \frac {2\pi \sqrt{-1}} {(b-1)!} \frac{\partial^b F} {\partial \zeta_1^b} (\xi_1, \zeta_2)
=b \oint \frac F {(\zeta_1-\xi_1)^{b+1}} d \zeta_1.  
\]

Since $Q=\widetilde{\zeta_2}(\zeta_1-\xi_1)$, 
 the identity \eqref{eqn:Hab} holds for $a=0$ with 
$
H_{0,b}=b \widetilde{\zeta_2}.  
$

For $a \ge 1$, 
 we proceed by induction on $b$.  
Obviously, 
 \eqref{eqn:Hab} holds
 for $b=0$
 with $H_{a,0}=0$.  
Suppose $a, b \ge 1$.  
By \eqref{eqn:230317100}, 
 one has
\[
  \frac{(\zeta_1-z)^a}{Q^b}
  =
  \frac{(\xi_1-z)^a}{Q^b}
 +\frac1 {\widetilde \zeta_2}
  \sum_{i=0}^{a-1}\frac{(\xi_1-z)^i (\zeta_1-z)^{a-1-i}}{Q^{b-1}}.  
\]
Since both $\xi_1$ and $\widetilde{\zeta_2}$ are independent
 of the variable $\zeta_1$, 
 the induction step for the identity \eqref{eqn:Hab} is justified 
 by the recurrence relation \eqref{eqn:Habdef} defining $H_{a,b}$.  
\end{proof}

Here are the first two examples of the family $H_{a,b}$
 for $b=1$ and $2$.  
\begin{align}
H_{a,1}=&t^a(\zeta_2-z)^a\widetilde{\zeta_2}^{1-a}, 
\label{eqn:23032224}
\\
H_{a,2}=&t^{a-1}(\zeta_2-z)^{a-1}\widetilde{\zeta_2}^{1-a}
(2t(\zeta_2-z)+a Q).  
\label{eqn:23032610}
\end{align}

The proof for Proposition \ref{prop:23032266} uses
 a special case of the formul{\ae} \eqref{eqn:Hab}:
\begin{align}
\label{eqn:23032244}
\oint \frac{(\zeta_1-z)^2}{Q^2} \frac{\partial F}{\partial \zeta_1}d \zeta_1
=\,&\oint \frac {2t(\zeta_1-z)(\zeta_2-z)F}{Q^3}d \zeta_1, 
\\
\label{eqn:23031838}
\oint \frac{\zeta_j}{Q} \frac{\partial F}{\partial \zeta_j} d \zeta_j
=\,&
-(Q -\frac{\partial Q}{\partial \zeta_j} \zeta_j)
\oint \frac{F}{Q^{2}} d \zeta_j
\qquad
\text{for $j=1,2$}.  
\end{align}

\begin{proof}
[Proof of Proposition \ref{prop:23032266}]
By the definition \eqref{eqn:Xi} of $\xi_1$ and $\xi_2$, 
 a direct computation shows
\begin{equation}
   \widetilde{\zeta_1} \widetilde{\zeta_2} \xi_1 = (z+t) Q -t^2 \zeta_1, 
\qquad
   \widetilde{\zeta_1} \widetilde{\zeta_2} \xi_2 = (z-t) Q -t^2 \zeta_2.  
\label{eqn:23032263}
\end{equation}

By \eqref{eqn:Hab} and \eqref{eqn:23032224}, 
 one has
\[
   \oint \frac{-\zeta_1 \zeta_2+ \zeta_2^2}{Q} 
         \frac{\partial F}{\partial \zeta_2} d \zeta_2
   =
   \xi_2 (\xi_2-\zeta_1)\widetilde{\zeta_1} 
   \oint \frac{F}{Q^2} d \zeta_2.  
\]
By \eqref{eqn:23031810} and \eqref{eqn:23032263}, 
 the right-hand side equals
\[
   -\xi_2(\zeta_1-z)^2 \oint \frac F{Q^2} d \zeta_2
  =-\oint 
    \frac{((z-t)Q-t^2\zeta_2) (\zeta_1-z)^2}{\widetilde{\zeta_1}\widetilde{\zeta_2} Q^2} 
 F d \zeta_2.  
\]
Applying this formula to $F:=\frac{\partial f}{\partial \zeta_1}$, 
 one has 
\begin{equation}
\label{eqn:zeta12Q}
   (2 \pi \sqrt{-1})^2 I_2(f)
  = - \oint \oint
    \frac{((z-t)Q-t^2\zeta_2) (\zeta_1-z)^2}{\widetilde{\zeta_1}\widetilde{\zeta_2} Q^2} 
   \frac{\partial f}{\partial \zeta_1}
   d \zeta_1 d \zeta_2.  
\end{equation}

Since 
$
   -\widetilde{\zeta_1} \widetilde{\zeta_2} (Q+t^2)+Q^2
   =t^4
$, 
 one has
\[
   \frac{1}{\widetilde{\zeta_2}Q^2}
  =-\frac{\widetilde{\zeta_1}}{t^4 Q^2}(Q+t^2)
   + \frac 1 {t^4\widetilde{\zeta_2}}.  
\]
Thus the function 
\[
  \frac{((z-t)Q-t^2\zeta_2)(\zeta_1-z)^2}{\widetilde{\zeta_1}}\cdot \frac 1 {t^4\widetilde{\zeta_2}} \frac{\partial f}{\partial \zeta_1}
  =
  \frac{\xi_2(\zeta_1-z)^2}{t^4}\frac{\partial f}{\partial \zeta_1}
\]
 is holomorphic function in $\zeta_1$, 
 and does not contribute
 to the integral in \eqref{eqn:zeta12Q}, 
 which reduces therefore to 
\begin{align*}
  &\oint \oint \frac{(Q+t^2)((z-t)Q-t^2\zeta_2)(\zeta_1-z)^2}{t^4 Q^2}
               \frac{\partial f}{\partial \zeta_1} 
               d \zeta_1 d \zeta_2
\\
   =
   & \oint \oint \frac{(t^2 Q (z-t-\zeta_2)-t^4 \zeta_2)(\zeta_1-z)^2}
                     {t^4 Q^2}
   \frac{\partial f}{\partial \zeta_1}
   d \zeta_1 d \zeta_2
\\
   =
   & -\oint \oint \frac{\widetilde{\zeta_2}(\zeta_1-z)^2}{t^2 Q} 
                 \frac{\partial f}{\partial \zeta_1} d \zeta_1 d \zeta_2
   - \oint \oint \frac{\zeta_2 (\zeta_1-z)^2}{Q^2} 
                 \frac{\partial f}{\partial \zeta_1} d \zeta_1 d \zeta_2
\\
   =
   &- \oint \oint \frac{(\zeta_2-z)^2}{Q^2} 
                 f d \zeta_1 d \zeta_2
   - \oint \oint \frac{2 t (\zeta_1-z)(\zeta_2-z)\zeta_2}{Q^3} 
                 f d \zeta_1 d \zeta_2.  
\end{align*}
In the first equality, 
 we have used the fact
 that the integral involving $Q^2$ 
 in the numerator vanishes.  
The last equality follows from Lemma \ref{lem:230326}, 
 or more precisely, from \eqref{eqn:23032244} and \eqref{eqn:23031838}.  
Hence the formula for $I_2(f)$ is proved.  
The proof for $I_1(f)$ is similar.  
\end{proof}

\begin{proof}
[Proof of Proposition \ref{prop:23031775}]
By \eqref{eqn:23031838}, 
\begin{align*}
\oint \frac{(\zeta_1-\zeta_2)\frac{\partial f}{\partial \zeta_1}} Q d \zeta_1 
=& (-Q+ \frac{\partial Q}{\partial \zeta_1} \zeta_1 - \frac{\partial Q}{\partial \zeta_1} \zeta_2)
\oint \frac f {Q^2} d \zeta_1
\\
\oint \frac{(\zeta_1-\zeta_2)\frac{\partial f}{\partial \zeta_2}} Q d \zeta_2 
=& (\frac{\partial Q}{\partial \zeta_2} \zeta_1 + Q - \frac{\partial Q}{\partial \zeta_2} \zeta_2)
\oint \frac f {Q^2} d \zeta_2.  
\end{align*}
Therefore one obtains 
\begin{equation*}
I_3(f)
=\frac{1}{(2 \pi \sqrt{-1})^2} 
  \oint\oint \frac{- (\zeta_1-z)^2-(\zeta_2-z)^2}{Q^2} f d \zeta_1 d \zeta_2
\end{equation*}
because 
$
  (-Q+\widetilde{\zeta_2} \zeta_1-\widetilde{\zeta_2}\zeta_2)
  -
  (\widetilde{\zeta_1} \zeta_1 + Q -\widetilde{\zeta_1}\zeta_2)
=
  - ({\zeta_1}-z)^2 -({\zeta_2}-z)^2.  
$

\end{proof}

By Propositions \ref{prop:23032266} and \ref{prop:23031775},
 the proof of Theorem \ref{thm:Luminy0316} is now complete.

We end this section by providing an example
 of closed formul\ae\ for $T f(z,t)$ for a specific family
 of functions $f \in {\mathcal{O}}(\Pi \times \Pi)$, 
 where $\Pi$ is the upper half plane.  
The family $\{f_{\ell}\}_{\ell \in {\mathbb{N}}}$ below
 gives the complete set 
 of \lq\lq{singular vectors}\rq\rq\ 
 in the tensor product of the two Hardy space, 
 see Section \ref{subsec:6.3}:
\begin{example}
\label{ex:23050410}
For $\ell \in {\mathbb{N}}$, 
 we set 
\begin{equation}
\label{eqn:minKK}
   f_{\ell}(\zeta_1, \zeta_2)
   :=(\zeta_1- \zeta_2)^{\ell}(\zeta_1 + \sqrt{-1})^{-\ell-1}
     (\zeta_2+\sqrt{-1})^{-\ell-1}.  
\end{equation}
Then one has the following:
\newline
{\rm{(1)}}\enspace
$P f_{\ell}=-\ell(\ell+1) f_{\ell}$.  
\newline
{\rm{(2)}}\enspace
$(T f_{\ell})(z,t)=\left({2 \ell} \atop \ell \right)t^{\ell}
(z+\sqrt{-1})^{-2\ell-2}$.  
\end{example}

\begin{proof}
(1)\enspace
We set $[\ell, b, c]:=(\zeta_1-\zeta_2)^{\ell}(\zeta_1+i)^{-b}(\zeta_2+i)^{-c}$. 
By a direct computation from the definition \eqref{eqn:23031847} of $P$, 
 one has 
\begin{multline*}
  P[\ell, b, c]= -\ell(\ell+1)[\ell, b, c]
+b c[\ell+2, b+1, c+1]
\\
+(\ell+1)b[\ell+1, b+1, c] -(\ell+1)c[\ell+1, b, c+1].  
\end{multline*}
Since $[\ell+1, b+1, c]-[\ell+1, b, c+1]=-[\ell+2, b+1, c+1]$, 
 we have $P[\ell, \ell+1, \ell+1]=-\ell(\ell+1)[\ell, \ell+1, \ell+1]$.  
\newline
(2)\enspace
By Corollary \ref{cor:Luminy0316} and Theorem \ref{thm:gRC}, 
 one has
\[
  (T f_{\ell})(z, t)= \frac{1}{\ell !}t^{\ell}(R_{\ell}f_{\ell})(z).  
\]
By the definition \eqref{eqn:RC} of the Rankin--Cohen bracket $R_{\ell}$, 
 one has
\begin{align*}
  (R_{\ell}f_{\ell})(z)
=&(R_{\ell} (\zeta_1-\zeta_2)^{\ell})(z + \sqrt{-1})^{-2 \ell -2}
\\
=& \frac{(2\ell)!}{\ell!}(z+\sqrt{-1})^{-2\ell-2},
\end{align*}
where the second equation follows from the formula
$$
\sum_{j=0}^\ell \left( \ell \atop j\right)^2=\frac{(2\ell)!}{\ell! \ell!}.
$$
Thus the second assertion is verified.  
\end{proof}

%%%%%%%%%%%%%%%%%%%%%%%%%%%%%%%%%%%%%%%%%%%%%%%%%%%%%%%%%
\section{Generating operators and holographic operators}
\label{sec:holographic}
%%%%%%%%%%%%%%%%%%%%%%%%%%%%%%%%%%%%%%%%%%%%%%%%%%%%%%%%%

Throughout this section, 
 we assume
 that $D$ is a convex domain in ${\mathbb{C}}$.  
Then any two elements $\zeta_1, \zeta_2 \in D$
 can be joined by a line segment contained in $D$.  
For $\ell \in {\mathbb{N}}$, 
 we consider a weighted average of $g \in {\mathcal{O}}(D)$
 along the line segment
 between $\zeta_1$ and $\zeta_2$
 given by 
\[
  (\Psi_{\ell} g)(\zeta_1, \zeta_2)
  :=(\zeta_1- \zeta_2)^{\ell}
    \int_{-1}^{1} 
      g(\frac{(\zeta_2-\zeta_1)v +(\zeta_1 + \zeta_2)}{2})(1-v^2)^{\ell}d v.  
\]

We investigate the \lq\lq{generating operator}\rq\rq\ $T$
 in connection with $\Psi_{\ell}$.  
Recall from Corollary \ref{cor:Luminy0316}
 that if $f \in \Sol(D \times D)_{\ell}$, 
 namely, 
 if $P f = -\ell(\ell+1) f$, 
 then $t^{-\ell}(T f)(z, t)$ is independent
 of $t$, 
 which we shall simply denote by 
 $(t^{-\ell} T f)(z)$.

\begin{theorem}
\label{thm:Tlsol}
Let $\ell \in {\mathbb{N}}$.  
\par\noindent
{\rm{(1)}}\enspace
$t^{-\ell}T \colon \Sol(D \times D)_{\ell} \overset \sim \rightarrow {\mathcal{O}}(D)$
 is a bijection.  
\par\noindent
{\rm{(2)}}\enspace
The inverse of $t^{-\ell}T$ is given 
 by the integral operator $\Psi_{\ell}$, 
 namely, 
$
   \Psi_{\ell} \colon {\mathcal{O}}(D)
   \overset \sim \rightarrow
   \Sol(D \times D)_{\ell}
$
 is a bijection
 and 
$
  t^{-\ell} T \circ \Psi_{\ell}
  = \frac{2^{2\ell+1}}{2 \ell+1} \operatorname{id}.
$
\end{theorem}

As an immediate consequence of Theorem \ref{thm:Tlsol} (2), 
 one has the following:

\begin{corollary}
\label{cor:dimSol}
For any $\ell \in {\mathbb{N}}$, 
 $\Sol({\mathbb{C}} \times {\mathbb{C}})_{\ell}$ is infinite-dimensional.  
\end{corollary}

\begin{proof}
[Proof of Theorem \ref{thm:Tlsol}]
First, 
 we prove
$
  \operatorname{Image}\Psi_{\ell} \subset \Sol(D \times D)_{\ell}.  
$
Recall from \eqref{eqn:23031847} the definition of $P.$  
A direct computation shows 
\begin{align*}
&P(\Psi_{\ell}g)
+ \ell (\ell+1) \Psi_{\ell}g 
\\
=&
 -\frac 1 2 (\zeta_1 - \zeta_2)^{\ell+1} \int_{-1}^1 \frac{\partial}{\partial v}
 (g'(\frac{(\zeta_2- \zeta_1)v+(\zeta_1 + \zeta_2)}2)(1-v^2)^{\ell+1}) d v, 
\end{align*}
which vanishes for any $g \in {\mathcal{O}}(D)$.  
Hence $\Psi_{\ell}g$ is an eigenfunction of $P$
 for the eigenvalue $-\ell(\ell+1)$.

Second, 
 we prove
 that the \lq\lq{generating operator}\rq\rq\ $T$ gives
 the inverse of $\Psi_{\ell}$
 up to scalar multiplication, 
 that is, 
\begin{equation}
\label{eqn:Tpsi}
   T(\Psi_{\ell}g)(z,t)=\frac{2^{2 \ell+1}}{2 \ell +1} t^{\ell}g(z)
\qquad\text{for any $g \in {\mathcal{O}}(D)$.  }
\end{equation}

To see \eqref{eqn:Tpsi}, 
 we observe from Corollary \ref{cor:Luminy0316}
 that $t^{-\ell}(T \Psi_{\ell}g)(z,t)$ does not depend 
 on the variable $t$
 because $\Psi_{\ell} g \in \Sol(D \times D)_{\ell}$.  
On the other hand, 
 it follows from the expansion \eqref{eqn:QTaylor}
 that the coefficient of $t^{\ell}$ in $(T f)(z,t)$
 is given by 
\[
  \frac{1}{(2 \pi \sqrt{-1})^2}
  \oint \oint \frac{(-1)^{\ell}(\zeta_1-\zeta_2)^{\ell} f(\zeta_1, \zeta_2)}{
(\zeta_1-z)^{\ell+1}(\zeta_2-z)^{\ell+1}} d \zeta_1 d \zeta_2.  
\]

Applying this to $f=\Psi_{\ell}g$, 
 one sees that $t^{-\ell}(T \Psi_{\ell} g)(z,t)$ is equal to 
\begin{align*}
  &\frac{(-1)^{\ell}}{(2 \pi \sqrt{-1})^2}
  \oint \oint 
  \frac
  {(\zeta_1-\zeta_2)^{2\ell}\int_{-1}^1 g(\frac{(\zeta_2- \zeta_1)v+(\zeta_1 + \zeta_2)}2)(1-v^2)^{\ell} d v}
  {(\zeta_1-z)^{\ell+1}(\zeta_2-z)^{\ell+1}} d \zeta_1 d \zeta_2
\\
=& \frac{(-1)^{\ell}}{(\ell !)^2}
   \left. \frac{\partial^{2\ell}}
               {\partial \zeta_1^{\ell} \partial \zeta_2^{\ell}}\right|_{\zeta_1=\zeta_2=z}
   ((\zeta_1-\zeta_2)^{2\ell}
    \int_{-1}^1 g(\frac{(\zeta_2- \zeta_1)v+(\zeta_1 + \zeta_2)}2)(1-v^2)^{\ell} d v).  
\end{align*}

An iterated use of the Leibniz rule develops the right-hand side
 as a sum of various derivatives, 
 among which the only non-vanishing term is 
\[ g(z)\frac{(2 \ell)!}{(\ell!)^2}
   \int_{-1}^1 (1-v^2)^{\ell} d v 
= \frac{2^{2\ell+1}}{2 \ell+1}g(z).  
\]
Thus we have shown \eqref{eqn:Tpsi}, 
 hence the injective morphism
 $t^{-\ell}T \colon \Sol(D \times D)_{\ell} \to {\mathcal{O}}(D)$
 is also surjective.

Finally, 
 let us show the surjectivity of $\Psi_{\ell}$.  
For any $f \in \Sol(D \times D)_{\ell}$, 
 there exists $g \in {\mathcal{O}}(D)$
 such that
$(T f)(z,t)=t^{\ell}g (z)$
 by Corollary \ref{cor:Luminy0316}.  
Since the right-hand side equals 
 $\frac{2\ell+1}{2^{2\ell+1}} T \Psi_{\ell}(g)$
  by \eqref{eqn:Tpsi}, 
 one has 
 $f=\frac{2\ell+1}{2^{2\ell+1}} \Psi_{\ell}g$
 because $T$ is injective.  
Thus the surjectivity of $\Psi_{\ell}$ is shown.  
\end{proof}

\begin{remark}
When $D$ is the upper half plane $\Pi$, 
 the integral operator $\Psi_{\ell}$ appeared
 in the study of the {\it{holographic transforms}}
 for the branching problem
 of infinite-dimensional representations
 of $SL(2,{\mathbb{R}})$.  
In this case, 
 the bijectivity of $\Psi_{\ell}$ was shown in \cite{KP20}
 by a different approach based on the representation theory.  
See Section \ref{sec:SL}.  
\end{remark}

%%%%%%%%%%%%%%%%%%%%%%%%%%%%%%%%%%%%%%%%%%%%%%%%%%%
\section{The generating operator $T$ and the Hardy space}
\label{sec:Hardy}
%%%%%%%%%%%%%%%%%%%%%%%%%%%%%%%%%%%%%%%%%%%%%%%%%%%

Let $\Pi$ be the upper half plane.  
As we have seen in Example \ref{ex:UD}, 
 the \lq\lq{generating operator}\rq\rq\
 $T \colon {\mathcal{O}}(\Pi \times \Pi) \to {\mathcal{O}}(U_{\Pi})$
 is well-defined
 where $U_{\Pi}=\{(z,t) \in {\mathbb{C}}^2: 2|t|< \operatorname{Im}z\}$.  
This section discusses how the generating operator $T$ acts
 on the tensor product 
 of two Hardy spaces.

We recall that the Hardy space
 on $\Pi$
 is a Hilbert space defined by 
\[
  {\bf{H}}(\Pi)
 =\{h \in {\mathcal{O}}(\Pi):
  \|h \|_{{\bf{H}}(\Pi)}^2 := \sup_{y>0} \int_{-\infty}^{\infty} |h(x+\sqrt{-1}y)|^2 d x < \infty\}.  
\]
Let ${\bf{H}}(\Pi \times \Pi)$ be the Hilbert completion
 ${\bf{H}}(\Pi) \widehat {\otimes} {\bf{H}}(\Pi)$
 of the tensor product
 of two Hardy spaces ${\bf{H}}(\Pi)$.  
Any holomorphic differential operator $P$ 
 acting on ${\mathcal{O}}(\Pi \times \Pi)$ induces a continuous operator
 on this Hilbert space.    
In turn,
 the eigenspace  
$
   {\bf{H}}(\Pi \times \Pi)_{\ell}:=\Sol(\Pi \times \Pi)_{\ell} \cap {\bf{H}}(\Pi \times \Pi)
$
 is a Hilbert subspace for every $\ell \in {\mathbb{N}}$.

\begin{theorem}
\label{thm:Hardy}
Let $P$ be the differential operator given in \eqref{eqn:23031847}.  
\par\noindent
{\rm{(1)}}\enspace
The differential operator $P$ defines
 a self-adjoint operator
 on the Hilbert space ${\bf{H}}(\Pi \times \Pi)$.  

\par\noindent
{\rm{(2)}}\enspace
{\rm{(Eigenspace decomposition)}}\enspace
${\bf{H}}(\Pi \times \Pi)$ decomposes
 into the discrete Hilbert sum
 of eigenspaces
$
   {\bf{H}}(\Pi \times \Pi)_{\ell}
$
 of $P$ where $\ell$ runs over ${\mathbb{N}}$.  

\par\noindent
{\rm{(3)}}\enspace
The generating operator $T$ induces a family of linear operators
\[
  t^{-\ell} T \colon 
  {\bf{H}}(\Pi \times \Pi)_{\ell}
  \overset\sim \rightarrow
  {\mathcal{O}}(\Pi) \cap
  L^2(\Pi, y^{2\ell} d x d y)
\]
which are unitary up to rescaling:
\begin{equation}
\label{eqn:Tlnorm}
  \|t^{-\ell} T f\|_{L^2(\Pi, y^{2\ell+2} d x d y)}^2
  =
  b_{\ell} \|f\|_{{\bf{H}}(\Pi \times \Pi)}^2
\quad
\text{for any $f \in {\bf{H}}(\Pi \times \Pi)_{\ell}$}  
\end{equation}
where we set 
\begin{equation}
\label{eqn:23041613}
   b_{\ell}
 :=\frac{(2 \ell)!}{2^{2\ell+2} \pi (2 \ell+1)(\ell !)^2}
  =\frac{(2 \ell-1)!!}{4 \pi (2 \ell+1)(2 \ell)!!}.  
\end{equation}

\end{theorem}

For the proof of Theorem \ref{thm:Hardy}, 
 we use the double Fourier--Laplace transform
 ${\mathcal{F}}$
 defined by 
\[
  F(x,y) \mapsto ({\mathcal{F}}F)(\zeta_1, \zeta_2)
:=\int_{0}^{\infty}\int_{0}^{\infty} F(x,y)e^{\sqrt{-1}(x \zeta_1 + y \zeta_2)} d x d y.  
\]
According to the Payley--Wiener theorem, 
 the Fourier--Laplace transform ${\mathcal{F}}$ establishes a bijection from
 $L^2({\mathbb{R}}_+ \times {\mathbb{R}}_+)$
 onto ${\bf{H}}(\Pi \times \Pi)$, 
 and satisfies 
$
\|{\mathcal{F}}F\|_{{\bf{H}}(\Pi \times \Pi)}^2
=
(2\pi)^2 \|F\|_{L^2({\mathbb{R}}_+ \times {\mathbb{R}}_+)}^2
$
for all $F \in L^2({\mathbb{R}}_+ \times {\mathbb{R}}_+)$.  
The inverse 
${\mathcal{F}}^{-1} \colon {\bf{H}}(\Pi \times \Pi) \to L^2({\mathbb{R}}_+ \times {\mathbb{R}}_+)$
 is given by 
\[
   ({\mathcal{F}}^{-1}f)(x,y)
 =\lim_{\eta_1 \downarrow 0}\lim_{\eta_2 \downarrow 0}
 \frac 1 {(2\pi)^2}
 \int_{{\mathbb{R}}^2} f(\zeta_1, \zeta_2) e^{-\sqrt{-1}(\zeta_1 x+ \zeta_2 y)} d \xi_1 d \xi_2, 
\]
where we write $\zeta_j=\xi_j + \sqrt{-1} \eta_j$.

The change of variables $(x,y)=(\frac s 2(1-v), \frac s 2(1+v))$
 yields a unitary map
$
  L^2({\mathbb{R}}_+ \times (-1,1), s d s d v)
  \overset \sim \longrightarrow
  L^2({\mathbb{R}}_+ \times {\mathbb{R}}_+, 2 d x d y).
$
We denote its composition with ${\mathcal{F}}$ by 
\[
   \widetilde{\mathcal{F}} \colon L^2({\mathbb{R}}_+ \times (-1,1), s d s d v)
\to {\bf{H}}(\Pi \times \Pi).  
\]
The inverse is given by
 $(\widetilde{\mathcal{F}}^{-1}f)(s,v)
 =({\mathcal{F}}^{-1}f)(\frac s 2(1-v), \frac s 2(1-v))$.

\begin{proposition}
\label{prop:Ftilde}
{\rm{(1)}}\enspace
$
  \widetilde{\mathcal{F}} \colon 
  L^2({\mathbb{R}}_+ \times (-1,1), s d s d v)
  \overset \sim \rightarrow
{\bf{H}}(\Pi \times \Pi)
$
is a unitary map
 up to a scalar multiplication, 
 namely, 
\[
  \|f\|_{{\bf{H}}^2(\Pi \times \Pi)}^2
  =
2\pi^2 \|(\widetilde{\mathcal{F}}^{-1} f)(s,v)\|_{ L^2({\mathbb{R}}_+ \times (-1,1), s d s d v)}^2
\quad
\text{for $f \in {\bf{H}}(\Pi \times \Pi)$}.
\]
\newline
{\rm{(2)}}\enspace
The operator $\widetilde P:= \widetilde {\mathcal{F}}^{-1} \circ P \circ {\mathcal{F}}$ takes the following form:
\begin{equation}
\label{eqn:Ptilde}
  \widetilde P=(1-v^2) \partial_v^2 -2v \partial_v.  
\end{equation}
\end{proposition}

\begin{proof}
(1)\enspace
For any $F(x,y)$, 
 one has 
$
  \|{\mathcal{F}} F\|_{{\bf{H}}(\Pi \times \Pi)}^2
  =
  (2 \pi)^2 \|F\|_{L^2({\mathbb{R}}_+ \times {\mathbb{R}}_+, d x d y)}^2
  =
  2 \pi^2 \|F(\frac s 2 (1-v), \frac s 2 (1+v))\|_{L^2({\mathbb{R}}_+ \times (-1,1), s d s d v)}^2.  
$
\newline
(2)\enspace
The Fourier transform ${\mathcal{F}}$ induces
 an isomorphism
 between the two Weyl algebras 
 ${\mathbb{C}}[\zeta_1, \zeta_2, \frac{\partial}{\partial \zeta_1}, \frac{\partial}{\partial \zeta_2}]$ and ${\mathbb{C}}[x, y, \frac{\partial}{\partial x}, \frac{\partial}{\partial y}]$
 by sending 
 $\frac{\partial}{\partial \zeta_1}$, $\frac{\partial}{\partial \zeta_2}$, 
 $\zeta_1$, and $\zeta_2$
 to $\sqrt{-1}x$, $\sqrt{-1}y$, 
 $\sqrt{-1}\frac{\partial}{\partial x}$, 
 and $\sqrt{-1}\frac{\partial}{\partial y}$, 
 respectively.  
In particular, 
 the holomorphic differential operator $P$
 in \eqref{eqn:23031847}
 is transformed
 into the operator 
$\widehat P
=(\partial_x-\partial_y)^2(x y)+(\partial_x-\partial_y)(x - y).  
$

By the change of variables $(x,y)=(\frac s 2(1-v), \frac s 2(1+v))$, 
 one has 
\[
  x-y=-s v, 
\qquad
 x y=\frac{s^2} 4(1-v^2), 
\qquad
  \partial_x - \partial_y = -\frac 2 s \frac{\partial}{\partial v}.  
\]
Hence the differential operator $\widehat P$ is transformed into $\widetilde P$. 
\end{proof}

\begin{proof}
[Proof of (1) and (2) in Theorem \ref{thm:Hardy}]
(1)\enspace
By Proposition \ref{prop:Ftilde}, 
 the differential operator $P$ is equivalent via $\widetilde {\mathcal{F}}$
 to the Legendre differential operator $\widetilde P$
 which does not involve the variable $s$.  
Since $\widetilde P$ defines a self-adjoint operator
 on $L^2({\mathbb{R}}_+ \times (-1,1), s d s d v)$, 
 see Fact \ref{fact:legendre} (2) in Appendix, 
 so does $P$ on ${\bf{H}}(\Pi \times \Pi)$
 via $\widetilde{\mathcal{F}}$.  
\newline
(2)\enspace
By \eqref{eqn:Ptilde}, 
 $P f = \lambda f$
 if and only if $\widetilde P(\widetilde{\mathcal{F}}^{-1}f)=\lambda(\widetilde{\mathcal{F}}^{-1} f)$.  
Hence $\widetilde{\mathcal{F}}$ induces an isomorphism
 $L^2({\mathbb{R}}_+, s d s) \otimes {\mathbb{C}} P_{\ell}(v)
\overset \sim \rightarrow {\bf{H}}(\Pi \times \Pi)_{\ell}$
  for every $\ell \in {\mathbb{N}}$, 
 where $P_{\ell}(v)$ is the $\ell$-th Legendre polynomial.  
Therefore the proof of the second statement
 is reduced to the classical theorem 
 that $\{P_{\ell}\}_{\ell \in {\mathbb{N}}}$
 forms an orthogonal basis
 in $L^2((-1,1), d v)$, 
 see Fact \ref{fact:legendre} (1).  
\end{proof}

To prove the third statement of Theorem \ref{thm:Hardy}, 
 we apply the \lq\lq{generating operator}\rq\rq\
 $T$ to 
the diagram below:
\begin{alignat}{5}
&{\mathcal{O}}(\Pi \times \Pi)
&&\,\,\supset\,\,
&&{\bf{H}}(\Pi \times \Pi)
&&\,\,\overset \sim {\underset{\widetilde{\mathcal{F}}}\leftarrow}\,\,
&&L^2({\mathbb{R}}_{+}, s d s) \widehat \otimes L^2(-1,1)
\notag
\\
& \hphantom{MM}\cup
&& 
&& \hphantom{MM}\cup
&&
&&\hphantom{MMMMM}\cup
\label{eqn:FSol}
\\
&\Sol(\Pi \times \Pi)_{\ell}
&&\,\,\supset\,\,
&&{\bf{H}}(\Pi \times \Pi)_{\ell}
&&\,\,\overset \sim {\underset{\widetilde{\mathcal{F}}}\leftarrow}\,\,
&&L^2({\mathbb{R}}_+, s d s) \otimes {\mathbb{C}} P_{\ell}(v).  
\notag
\end{alignat}
We recall
 that the weighted Bergman space is defined by 
\[
  {\bf{H}}^2(\Pi)_{\lambda}:={\mathcal{O}}(\Pi) \cap L^2(\Pi, y^{\lambda-2} d x d y) 
\]
for $\lambda >1$.  
We also recall some basic properties
 of the Fourier--Laplace transform
 of one variable 
 $\varphi(\xi) \mapsto ({\mathcal{F}}_{\mathbb{R}} \varphi)(z):=\int_{0}^{\infty}\varphi(\xi)e^{\sqrt{-1} z \xi}d \xi$.  
By the Plancherel formula, 
 one has 
\[
  \int_{\mathbb{R}}|{\mathcal{F}}_{\mathbb{R}} \varphi(x+\sqrt{-1}y)|^2 d x
  =
  2 \pi \int_0^{\infty}|\varphi(\xi)|^2e^{-2 y \xi}d \xi.  
\]
Integrating the both-hand sides against the measure
 $y^{\lambda-2}d y$, 
one obtains 
\begin{equation}
\label{eqn:FBlmd}
  \|{\mathcal{F}}_{\mathbb{R}} \varphi\|_{{\bf{H}}^2(\Pi)_{\lambda}}^2
=2^{2-\lambda}\pi \Gamma(\lambda-1)\|\varphi\|_{L^2({\mathbb{R}}_+, \xi^{1-\lambda}d \xi)}^2. 
\end{equation}
Thus ${\mathcal{F}}_{\mathbb{R}}$ gives
 a bijection from $L^2({\mathbb{R}}_+, \xi^{1-\lambda} d \xi)$
 onto ${\bf{H}}^2(\Pi)_{\lambda}$, 
 see \cite[Thm.\ XIII.1.1]{xFaKo94} for details.  

We show the following:

\begin{proposition}
\label{prop:IHES230326}
Let 
 $c_{\ell}:=\frac{(-1)^{\frac 3 2 {\ell}}}{(2\ell+1) \ell !}$ and 
\[
  T^{\mathcal{F}}(h(z) P_{\ell}(v))
  :=
  c_{\ell}\, h(\xi)\, \xi^{\ell+1}\, t^{\ell}.  
\]
Then the following diagram commutes.  
\begin{alignat*}{3}
&\hphantom{M}{\bf{H}}(\Pi \times \Pi)_{\ell}
&&\,\,\underset{\widetilde {\mathcal{F}}}{\overset \sim \leftarrow}\,\,
&&L^2({\mathbb{R}}_+, s d s) \otimes {\mathbb{C}}P_{\ell}(v)
\\
& \hphantom{MM}T\,\, \rotatebox[origin=c]{270}{$\overset \sim \rightarrow$}
&&
&&\hphantom{MMMM}\rotatebox[origin=c]{270}{$\overset \sim \rightarrow$}\,\,
T^{\mathcal{F}}
\\
&
{\bf{H}}^2(\Pi)_{2+2\ell}
\otimes
{\mathbb{C}} t^{\ell}
&&\,\,\underset{{\mathcal{F}}_{\mathbb{R}} \otimes \operatorname{id}}{\overset \sim \leftarrow}\,\,
&&L^2({\mathbb{R}}_+, \xi^{-1-2\ell} d \xi) \otimes {\mathbb{C}} t^{\ell}
\end{alignat*}
\end{proposition}

\begin{proof}
Take any $h \in L^2({\mathbb{R}}_+, s d s)$.  
Since $\widetilde{\mathcal{F}}(h P_{\ell})
 \in {\bf{H}}(\Pi \times \Pi)_{\ell}$
 by Proposition \ref{prop:Ftilde} and its proof, 
 one has
\[
   \ell ! T(\widetilde{\mathcal{F}}(h P_{\ell}))(z,t)
=
t^{\ell} R_{\ell}
\widetilde{\mathcal{F}}(h P_{\ell})(z)
\]
by Theorem \ref{thm:gRC} and 
Corollary \ref{cor:Luminy0316}.  
By the definition of $\widetilde{\mathcal{F}}$, 
 one has
\[
  \widetilde{\mathcal{F}}(h P_{\ell})(\zeta_1, \zeta_2)
=
  \frac 1 2 
  \int_{0}^{\infty} \int_{-1}^1
  h(z) P_{\ell}(v)
  G(s, v; \zeta_1, \zeta_2)
  s d s d v, 
\]
 where
$
   G(s, v;\zeta_1, \zeta_2):=
 e^{\sqrt{-1} \frac s 2 ((1-v)\zeta_1+(1+v)\zeta_2)}.  
$
The Legendre polynomials $P_{\ell}(v)$
 and the Rankin--Cohen bidifferential operators 
$
   R_{\ell} \colon {\mathcal{O}}({\mathbb{C}}^2) \to {\mathcal{O}}({\mathbb{C}})$
 are related to each other via the function $G$ as follows:
\begin{align*}
R_{\ell}G(s, v;\zeta_1, \zeta_2)
=&
\sum_{j=0}^{\ell}
(-1)^j \left(\ell  \atop j \right)^2
\left. \frac{\partial^{\ell}}{\partial\zeta_1^{\ell-j} \partial \zeta_2^j}\right|_{\zeta_1=\zeta_2=z}
G(s, v, \zeta_1, \zeta_2)
\\
=& (-1)^{\frac 3 2 \ell} e^{\sqrt{-1}zs} s^{\ell} P_{\ell}(v)
\end{align*}
for all $\ell \in {\mathbb{N}}$.  
Here we have used 
 the Rodrigues formula \eqref{eqn:Pl}
 for  the second equality.

By using the formula for the $L^2$-norm $P_{\ell}$
 (see Fact \ref{fact:legendre} (1)), 
 one has
\begin{align*}
R_{\ell}\widetilde{\mathcal{F}}(h P_{\ell})(z)
=&\frac{(-1)^{\frac 3 2 \ell}}2 \int_0^{\infty}\int_{-1}^1
 h(s) P_{\ell}(v)^2 e^{\sqrt{-1}z s} s^{\ell+1} d s d v
\\
=&\frac{(-1)^{\frac 3 2 \ell}}{2\ell+1}{\mathcal{F}}_{\mathbb{R}}(h(s)s^{\ell+1})(z).  
\end{align*}
Hence Proposition \ref{prop:IHES230326} is proved.  
\end{proof}

\begin{proof}
[Proof of (3) in Theorem \ref{thm:Hardy}]
In light of the isomorphism
\[
\widetilde{\mathcal{F}} \colon L^2({\mathbb{R}}_+, s d s) \otimes {\mathbb{C}}P_{\ell}(v) \overset \sim \longrightarrow {\bf{H}}(\Pi \times \Pi)_{\ell}, 
\]
 we take $h \in L^2({\mathbb{R}}_+, s d s)$
 and set $f :=\widetilde{\mathcal{F}}(h P_{\ell}) \in  {\bf{H}}(\Pi \times \Pi)_{\ell}$.  
By Proposition \ref{prop:Ftilde}, 
 one has
\begin{equation}
\label{eqn:Tl1}
\|f\|_{ {\bf{H}}(\Pi \times \Pi)}^2
=2\pi^2 \|h\|_{L^2({\mathbb{R}}_+, s d s)}^2 \|P_{\ell}\|_{L^2(-1,1)}^2
=\frac{4\pi^2}{2 \ell+1}\|h\|_{L^2({\mathbb{R}}_+, s d s)}^2.  
\end{equation}
Applying \eqref{eqn:FBlmd} to 
$
   \varphi:=t^{-\ell}T^{{\mathcal{F}}}(h P_{\ell})
$
 with $\lambda=2\ell+2$, 
 one has from Proposition \ref{prop:IHES230326}
 that 
\begin{align}
\notag
\|t^{-\ell} T f\|_{{\bf{H}}^2(\Pi)_{2+2\ell}}^2
=&2^{-2\ell} \pi (2 \ell)! \|t^{-\ell} T^{\mathcal{F}} (h P_{\ell})\|_{L^2({\mathbb{R}}_+, \xi^{-1-2\ell} d \xi)}^2 
\\
\label{eqn:Tl2}
=&\frac{\pi (2 \ell)!}{2^{2 \ell}(2\ell+1)^2(\ell !)^2}\|h\|_{L^2({\mathbb{R}}_+, s d s)}^2.  
\end{align}
It follows from \eqref{eqn:23041613}, \eqref{eqn:Tl1} and \eqref{eqn:Tl2}
 that 
\[
  \|t^{-\ell} T f\|_{{\bf{H}}^2(\Pi)_{2+2\ell}}^2
  =
  \frac{(2\ell-1)!!}{4 \pi (2 \ell+1) (2\ell)!!}
  \|f\|_{{\bf{H}}(\Pi \times \Pi)}^2
  =
  b_{\ell} \|f\|_{{\bf{H}}(\Pi \times \Pi)}^2.   
\]
Hence the third statement of Theorem \ref{thm:Hardy} is proved.  
\end{proof}

%%%%%%%%%%%%%%%%%%%%%%%%%%%%%%%%%%%%%%%
\section{Representation theory and the generating operator $T$}
\label{sec:SL}
%%%%%%%%%%%%%%%%%%%%%%%%%%%%%%%%%%%%%%%

If $D$ is simply connected, 
 then the group $\operatorname{Aut}(D)$
 of biholomorphic diffeomorphisms 
 acts transitively on $D$.  
This section discusses different perspectives
 of our generating operator $T$ from the viewpoint
 of the automorphism group of the domain, 
 in particular, from the (infinite-dimensional) representation theory of real reductive groups.  
Lie theory reveals structures
 of the generating operator $T$
 that are not otherwise evident.

\subsection {Normal derivatives and the generating operator $T$}
~~~
\newline
Let $\pi$ be an irreducible representation of a group $G$, 
 and $G'$ a subgroup.  
The $G$-module $\pi$ may be seen 
 as a $G'$-module
 by restriction, 
 for which we write $\pi|_{G'}$.  
For an irreducible representation $\rho$ 
 of the subgroup $G'$, 
 a {\it{symmetry breaking operator}}
 (SBO for short)
 is an intertwining operator from $\pi|_{G'}$
 to $\rho$, 
 whereas a {\it{holographic operator}} is an intertwining operator from 
$\rho$ to $\pi|_{G'}$.  
Suppose that the representations $\pi$ and $\rho$
 are geometrically defined, 
 {\it{e.g., }}
 they are realized
 in the spaces $\Gamma(X)$ and $\Gamma(Y)$
 of functions on a $G$-manifold $X$
 and its $G'$-submanifold $Y$, 
 respectively, 
 or more generally, 
 in the spaces of sections for some equivariant vector bundles.

When the restriction $\pi|_{G'}$
 is discretely decomposable
 \cite{xkdecomp}, 
 one may expect
 that taking \lq\lq{normal derivatives}\rq\rq\ 
 with respect to the submanifold $Y \hookrightarrow X$
 would yield SBOs.  
However, 
 this is not the case
even for the irreducible decomposition ({\it{fusion rule}})
 of the tensor product of two representations of $SL(2,{\mathbb{R}})$.  
See \cite[Thm.\ 5.3]{KP16} for more general cases.  
The underlying geometry for the fusion rule 
 of the Hardy spaces
 ${\bf{H}}(\Pi)$
 is given by a diagonal embedding
 of $Y=\Pi$ into $X:=Y \times Y$.  
Instead of using $X=\Pi \times \Pi$, 
 we consider $\widetilde X:=U_{\Pi}$
 as in Example \ref{ex:UD}.  
In this case 
 the \lq\lq{normal derivative}\rq\rq\
 of $\ell$-th order
 with respect to $Y \hookrightarrow \widetilde X$
 is given simply by 
\[
   N_{\ell}:=\operatorname{Rest}_{t=0} \circ (\frac{\partial}{\partial t})^{\ell}.  
\]
A distinguishing feature
 of the generating operator $T$
 is that all the normal derivatives
 $N_{\ell}$ give rise to symmetry breaking operators
 after the transformation by $T$, 
 symbolically written in the following diagram
 (see \eqref{eqn:pilmd} for the notation $\pi_{\lambda}$):
\begin{alignat*}{4}
&{\mathcal{O}}(X) 
&&\hphantom{MM}\overset{T}\hookrightarrow 
&&\hskip 1pc {\mathcal{O}}(\widetilde X)
\\
&\text{SBO}\searrow
&&\hphantom{MM}\circlearrowright
&&\swarrow \text{$\ell$-th normal derivative $N_{\ell}$}
\\
& &&(\pi_{2+2\ell},{\mathcal{O}}(Y)) &&
\end{alignat*}

%%%%%%%%%%%%%%%%%%%%%%%%%%%%%%%%%%%%%%%%%%%%%%%%%%%%%%%%%%%%%%%%%%%%%
\subsection{Modular forms and the generating operator $T$}
%%%%%%%%%%%%%%%%%%%%%%%%%%%%%%%%%%%%%%%%%%%%%%%%%%%%%%%%%%%%%%%%%%%%%

The Rankin--Cohen brackets were introduced in \cite{xCo75, xRa56}
 to construct holomorphic modular forms
 of higher weight from those of lower weight.  
This section highlights the relationship
 of our generating operator $T$ in \eqref{eqn:Luminy230314}
 and modular forms.

By Theorem \ref{thm:gRC}, 
 one has
\begin{equation}
\label{eqn:NTRl}
   N_{\ell} \circ T = R_{\ell},
\end{equation}
where $R_{\ell}$ are the Rankin--Cohen brackets \eqref{eqn:RC}.  
Then by a direct computation \cite{xCo75}
or by the F-method \cite{KP16}, 
 one sees the following covariance property:

\begin{proposition}
\label{prop:TSL}
For all $\ell \in {\mathbb{N}}$, 
 for any $g = \begin{pmatrix} a & b \\ c & d\end{pmatrix} \in SL(2,{\mathbb{R}})$
 and for any $f \in {\mathcal{O}}(\Pi \times \Pi)$, 
 one has
\[
N_{\ell} \circ (T f^g)(z)=(c z + d)^{-2\ell -2}((N_{\ell} \circ T) f)(\frac{a z +b}{c z +d})
\]
where 
$f^g(\zeta_1, \zeta_2):=(c\zeta_1+d)^{-1}(c\zeta_2+d)^{-1}f(\frac{a \zeta_1 + b}{c \zeta_1+d}, \frac{a \zeta_2 + b}{c \zeta_2+d})$.  
\end{proposition}

To clarify its representation-theoretic meaning, 
 we write $\pi_{\lambda}$ $(\lambda \in {\mathbb{Z}})$
 for a representation of $SL(2,{\mathbb{R}})$
 on ${\mathcal{O}}(\Pi)$
 given by 
\begin{equation}
\label{eqn:pilmd}
  \pi_{\lambda}(g)h(z)
  =
  (c z + d)^{-\lambda}
  h(\frac {a z+b}{c z+d})
\qquad
 \text{for $g^{-1}=\begin{pmatrix} a & b \\ c & d\end{pmatrix}$}.  
\end{equation}
Then Proposition \ref{prop:TSL} tells us
 that 
\begin{equation}
\label{eqn:modularity}
 (N_{\ell} \circ T)\circ(\pi_1(g) \boxtimes \pi_1(g))
 =
 \pi_{2\ell+2}(g) \circ (N_{\ell}\circ T)
\end{equation}
for any $g \in SL(2,{\mathbb{R}})$.  
Therefore, 
 for a subgroup $\Gamma$, 
 $N \circ T(f)$ is $\Gamma$-invariant
 whenever $f$ is $(\Gamma \times \Gamma)$-invariant.

Suppose that $\Gamma$ is a congruence subgroup of $SL(2,{\mathbb{Z}})$.  
For any modular form $h$
 of level $\Gamma$ and weight $1$, 
 we set 
\[
  H(z,t):=
  \frac{1}{(2 \pi \sqrt{-1})^2}
  \oint_{C_1} \oint_{C_2} \frac{h(\zeta_1)h(\zeta_2)}{(\zeta_1-z)(\zeta_2-z)+t(\zeta_1-\zeta_2)} d \zeta_1 d \zeta_2.  
\]
It follows from \eqref{eqn:NTRl}
 that 
$
   (N_{\ell}H)(z)
  =\left. (\frac{\partial}{\partial t})^{\ell}\right|_{t=0}H(z,t)
  =R_{\ell}(h(\zeta_1) h(\zeta_2))(z)
$
 is a modular form
 of level $\Gamma$ and weight $2 \ell +2$ for all $\ell \in {\mathbb{N}}$.

%%%%%%%%%%%%%%%%%%%%%%%%%%%%%%%%%%%%%%%%%%%%%%%%%%%%%%%%%%%%%%%%%%%%%
\subsection{Unitary representation and the generating operator $T$}
\label{subsec:6.3}
%%%%%%%%%%%%%%%%%%%%%%%%%%%%%%%%%%%%%%%%%%%%%%%%%%%%%%%%%%%%%%%%%%%%%
Viewed as a representation 
of the universal covering group $SL(2,{\mathbb{R}})\widetilde{\hphantom{i}}$, 
 the representation $\pi_{\lambda}$ is well-defined
 for all $\lambda \in {\mathbb{C}}$.  
For $\lambda>1$, 
 $\pi_{\lambda}$ leaves the weighted Bergman space
$
{\bf{H}}^2(\Pi)_{\lambda}
 =
{\mathcal{O}}(\Pi) \cap L^2(\Pi, y^{\lambda-2}d x d y)
$
 invariant, 
 and $SL(2,{\mathbb{R}})\tilde{\,\,}$ acts
 as an irreducible unitary representation
 on the Hilbert space ${\bf{H}}^2(\Pi)_{\lambda}$.  
These unitary representations $(\pi_{\lambda}, {\bf{H}}^2(\Pi)_{\lambda})$
 are referred 
 to as (relative) {\it{holomorphic discrete series representations}}
 of $SL(2,{\mathbb{R}})\tilde{\,\,}$. 
In particular, 
 the set of holomorphic discrete series representations
 of the group $P S L(2,{\mathbb{R}}) = S L(2,{\mathbb{R}})/\{\pm I_2\} \simeq \operatorname{Aut}(\Pi)$
 is given by $\{\pi_{\lambda}:\lambda=2,4,6,\dots \}$.

If $\lambda=1$
 then ${\bf{H}}^2(\Pi)_{\lambda}=\{0\}$, 
 however, 
 the Hardy space ${\bf{H}}(\Pi)$ is an invariant subspace
 of $(\pi_{\lambda}, {\mathcal{O}}(\Pi))$ 
 with $\lambda=1$, 
 and $SL(2,{\mathbb{R}})$ acts  on ${\bf{H}}(\Pi)$ as an irreducible unitary representation, 
 too.

With these notations, 
 one may interpret Theorem \ref{thm:Hardy} 
 as a decomposition
 of the completed tensor product
 of two copies of the unitary representation
 $(\pi_1, {\bf{H}}(\Pi))$
 on the Hardy space
 into a multiplicity-free discrete sum
 of irreducible unitary representations:
\[
   {\bf{H}}(\Pi) \widehat \otimes {\bf{H}}(\Pi)
   \simeq
   {\overset{\infty}{\underset{\ell=0}\sum}}^{\lower1ex\hbox{\footnotesize{$\oplus$}}} {\bf{H}}^2 (\Pi)_{2+2\ell}
\qquad
\text{(Hilbert direct sum)}.
\]
The right-hand side may be thought
 of as a \lq\lq{model}\rq\rq\
 of holomorphic discrete series representations
 of $PSL(2,{\mathbb{R}})$
 in the sense 
that all such representations occur exactly once.

\subsection{Limit of the weighted Bergman spaces}
~~~
\newline
The Hardy norm 
 $\|\cdot\|_{{\bf{H}}(\Pi)}$ may be regarded
 as the residue
 of the analytic continuation
 of the norm
 of the weighted Bergman space
 ${\bf{H}}^2(\Pi)_{\lambda}$
 which is originally defined 
 for real $\lambda>1$:
\[
   \|\cdot\|_{{\bf{H}}(\Pi)}^2
   =\lim_{\lambda \downarrow 1}
    (\lambda-1)
    \|\cdot\|_{{\bf{H}}^2(\Pi)_{\lambda}}^2.  
\]

Then the exact formula \eqref{eqn:Tlnorm} in Theorem \ref{thm:Hardy} 
 may be thought
 of as the limit of \cite[Thm.\ 2.7]{KP20}
 which dealt with the weighted Bergman spaces, 
 namely, 
 our $b_{\ell}$ in Theorem \ref{thm:Hardy}
 may be rediscovered by the following limit procedure
 with the notation 
 as in \cite[(2.3) and (2.4)]{KP20}:

\begin{align*}
& \frac {1}{(\ell !)^2} \lim_{\lambda' \downarrow 1}\lim_{\lambda'' \downarrow 1}
\frac{c_{\ell}(\lambda',\lambda'')r_{\ell}(\lambda',\lambda'')}{(\lambda'-1)(\lambda''-1)}
\\
=& \frac {1}{(\ell !)^2}
\lim_{\lambda' \downarrow 1}\lim_{\lambda'' \downarrow 1}
\frac{\Gamma(\lambda'+\ell) \Gamma(\lambda''+\ell)}{(\lambda'+\lambda''+2\ell-1)\Gamma(\lambda'+\lambda''+\ell-1)\ell!} 
\cdot
\frac{\Gamma(\lambda'+\lambda''+2\ell-1)}
{2^{2\ell+2} \pi \Gamma(\lambda')\Gamma(\lambda'')}
\\
=& \frac{(2\ell)!}{(2\ell+1)\pi(\ell !)^2 2^{2\ell+2}}
=\frac{(2\ell-1)!!}{4 \pi(2\ell+1) (2\ell)!!}
=b_{\ell}.  
\end{align*}

%%%%%%%%%%%%%%%
\section{Appendix: The Legendre polynomials}
\label{sec:appendix}
%%%%%%%%%%%%%%%%
Suppose $\ell \in {\mathbb{N}}$.  
The Legendre polynomial $P_{\ell}(v)$ is a polynomial solution
 to the  Legendre differential equation:

\[
((1-v^2)\frac{d^2}{dv^2}-2v\frac{d}{dv}+\ell(\ell+1))f=0
\]
 which is normalized by 
$
  P_{\ell}(1)=1.  
$
Then it satisfies the Rodrigues formula
\begin{equation}
\label{eqn:Pl}
  P_{\ell}(v):=
  \frac 1 {2^{\ell}} \sum_{j=0}^{\ell} 
  \left(\ell \atop j\right)^2(v-1)^{\ell-j} (v+1)^j.  
\end{equation}

\begin{fact}
[{\cite{xLe785, xSrMa84}}]
\label{fact:legendre}
{\rm{(1)}}\enspace
The Legendre polynomials 
 $\{P_{\ell}(v)\}_{\ell \in {\mathbb{N}}}$ form 
 an orthogonal basis
 in the Hilbert space $L^2((-1,1), d v)$
 with the following norm:
\[
 (P_{\ell}, P_{\ell'})_{L^2(-1,1)} = \frac 2 {2\ell+1} \delta_{\ell\ell'}.  
\]

\par\noindent
{\rm{(2)}}\enspace
The differential operator $(1-v^2)\frac{d^2}{dv^2}-2v\frac{d}{dv}$
 is essentially self-adjoint 
 on the Hilbert space $L^2((-1,1), d v)$.  
\end{fact}

The Legendre polynomials $P_{\ell}(x)$ are particular cases of the Jacobi polynomials
 $P_{\ell}^{\alpha, \beta}(x)$
 with $\alpha=\beta=0$.  

\vskip 1pc
\small{
\noindent{\bf{Acknowledgements.}}\enspace
The first named author was partially supported
 by the JSPS
 under the Grant-in Aid for Scientific Research (A) 
 (JP18H03669, JP23H00084).  
Both authors are grateful to the
Centre International de Rencontres Math{\'e}\-matiques
 (Luminy, France)
 and the
 Institut des Hautes {\'E}tudes Scientifiques
 (Bures-sur-Yvette, France), 
 where an important part of this work was done.  
\vskip 3pc

\leftline
{Toshiyuki KOBAYASHI}
\leftline
{Graduate School of Mathematical Sciences, }
\leftline
{The University of Tokyo,
 3-8-1 Komaba, Meguro, 
Tokyo, 153-8914, Japan} 

\&

\leftline
{French-Japanese Laboratory
 in Mathematics and its Interactions, }
\leftline{
FJ-LMI CNRS IRL2025, Tokyo, Japan}

\leftline{ E-mail:
\texttt{toshi@ms.u-tokyo.ac.jp}}

\vskip 1pc
\leftline
{Michael PEVZNER}
\leftline
{LMR, 
Universit{\'e} de Reims-Champagne-Ardenne, 
CNRS UMR 9008, F-51687,}
\leftline{Reims, France} 

\&

\leftline
{French-Japanese Laboratory
 in Mathematics and its Interactions, }
\leftline{
FJ-LMI CNRS IRL2025, Tokyo, Japan}
\leftline{E-mail: \texttt{ pevzner@math.cnrs.fr}}
\end{document}